\documentclass[a4paper,11pt]{amsart}

\usepackage{amsmath, amsthm, amsfonts, amssymb}
\usepackage[bookmarks=false]{hyperref}
\usepackage{enumerate}

\usepackage{color}

\numberwithin{equation}{section}

\newtheorem*{mainthm}{Main theorem}

\newtheorem{thm}{Theorem}[section]
\newtheorem{lem}[thm]{Lemma}

\newtheorem{cor}[thm]{Corollary}
\newtheorem{prop}[thm]{Proposition}

\theoremstyle{definition}
\newtheorem{rem}[thm]{Remark}

\newtheorem{example}[thm]{Example}

\newtheorem{df}[thm]{Definition}

\newcommand{\C}{\mathbf{C}}
\newcommand{\Z}{\mathbf{Z}}

\newcommand{\B}{\mathbf{B}}

\newcommand{\Ad}{\operatorname{Ad}}

\newcommand{\III}{{\rm III}}

\begin{document}

\title[On the weak relative Dixmier property]{On the weak relative Dixmier property}

\address{Research Institute for Mathematical Sciences, Kyoto University, Kyoto, 606-8502 Japan}

\author{Amine Marrakchi}
\email{amine.marrakchi@math.u-psud.fr}

\thanks{The author is supported by JSPS}

\subjclass[2010]{46L10, 46L36, 46L40, 46L55}


\begin{abstract}
We show that every inclusion of von Neumann algebras with a faithful normal conditional expectation has the weak relative Dixmier property. This answers a question of Popa \cite{Po99}. The proof uses an improvement of Ellis' lemma for compact convex semigroup.
\end{abstract}

\maketitle






\section{Introduction}

Let $N \subset M$ be an inclusion of von Neumann algebras. Following \cite[Definition 1.1]{Po99}, we say that the inclusion $N \subset M$ has the \emph{weak relative Dixmier property} if for every $x \in M$, the weak$^*$ closed convex hull of $\{ uxu^* \mid u \in \mathcal{U}(N) \}$ intersects $N' \cap M$. A classical result of Dixmier asserts that this property is always satisfied when $N=M$.  However, as explained in \cite{HP17}, if one does not make any assumption on the inclusion $N \subset M$, then the weak relative Dixmier property is not necessarily satisfied and can be very subtle as demonstrated by the following two examples:
 \begin{itemize}
 \item The inclusion $M \subset \B(H)$ has the weak relative Dixmier property if and only if $M$ is AFD. The if direction was first noticed by Schwartz \cite{Sc63} while the converse follows from Connes' celebrated work \cite{Co75}. 
 \item A type $\III_1$ factor $M$ has a trivial bicentralizer if and only if the inclusion $N \subset M$ has the weak relative Dixmier property, where $N$ is the \emph{continuous core} of $M$. This equivalence is due to Haagerup \cite{Ha85}. It is still a deep open problem to determine whether every type $\III_1$ factor has a trivial bicentralizer.
 \end{itemize}

On the other hand, it is an elementary fact \cite[Lemma 1.2]{Po99} that an inclusion $N \subset M$ has the weak relative Dixmier property if $M$ is tracial or more generally if there exists a faithful normal state $\varphi$ on $M$ such that $N \subset M_\varphi :=\{a \in M \mid \forall x \in M, \:  \varphi(a x)=\varphi( x a) \}$. This kind of averaging argument is actually a fundamental tool in von Neumann algebra theory. For example, it plays a key role in \cite{Po81} or in Popa's intertwining theory \cite{Po01, Po03}. In \cite[Remark 1.5]{Po99} (see also \cite[Problem 7]{HP17}), it was suggested that the weak relative Dixmier property might more generally hold for any inclusion \emph{with expectation} (this would in particular cover \cite[Corollary 1.4]{Po99}). Our main theorem solves this question.

\begin{mainthm} \label{main}
Let $N \subset M$ be an inclusion of von Neumann algebras. Suppose that there exists a faithful normal state $\varphi$ on $M$ such that $N$ is the range of a $\varphi$-preserving conditional expectation $E : M \rightarrow N$. Then the inclusion $N \subset M$ has the weak relative Dixmier property.
\end{mainthm}
The difficulty in the proof of this theorem arises only when $N$ is of type $\III_1$. Indeed, in all other cases, one can use discrete decompositions in order to reduce the problem to the tracial case. To deal with the remaining type $\III_1$ case, we use a recent result of the author \cite{Ma18} asserting the existence of an irreducible AFD subfactor with expectation inside $N$. We combine it with a new averaging argument for compact convex semigroups based on Ellis' lemma and inspired by Sinclair's proof of the existence of the injective enveloppe \cite{Si15}.

\section{Compact convex semigroups}
\begin{df}
A \emph{compact semigroup} $K$ is a non-empty compact space equipped with a semigroup operation such that $K \ni a \mapsto ab \in K$ is continuous for every given $b \in K$.
\end{df} 

The proof of the following result, known as \emph{Ellis' lemma}, is very short but it has many surprising applications to combinatorics \cite{FK89}. Recently, Sinclair applied it to give a very short proof of the existence of the injective enveloppe of an operator system \cite{Si15}.

Recall that an element $e$ of a semigroup $S$ is \emph{idempotent} if $ee=e$. We define an order relation between idempotents as follows: $e < f$ if and only if $ef=fe=e$.

\begin{thm}[\cite{FK89}] \label{Ellis} Let $K$ be a compact semigroup. Then $K$ contains an idempotent which is minimal for the order relation $<$.
\end{thm}
\begin{rem}
It is also shown in \cite{FK89} that an idempotent $e \in K$ is minimal if and only if $Ke$ is a minimal closed left ideal.
\end{rem}

\begin{df}
A \emph{compact convex semigroup} $K$ is a non-empty compact convex space (in some locally convex topological vector space) equipped with a semigroup operation such that $K \ni a \mapsto ab \in K$ is affine and continuous for every given $b \in K$.
\end{df}

\begin{example}
Let $M$ be a von Neumann algebra and let $\mathrm{UCP}(M)$ be the set of all UCP maps from $M$ into $M$, equipped with the topology of pointwise weak$^*$ convergence. Then $\mathrm{UCP}(M)$ is a compact convex semigroup.
\end{example}

We now strengthen the conclusion of Theorem \ref{Ellis} for compact convex semigroups.

\begin{thm} \label{minimal idempotent}
Let $K$ be a compact convex semigroup. Let $e \in K$ be a minimal idempotent. Then we have $exe=e$ for all $x \in K$.
\end{thm}
\begin{proof}
Let $I=Ke$ be the minimal closed left ideal generated by $e$. Take any $x \in K$ and let $J=\{ \frac{1}{2}(bxe+b) \mid b \in I \}$. Then $J \subset I$ and $J$ is a closed left ideal of $K$. By minimality of $I$, we must have $J=I$. Let $a$ be an extremal point of the convex compact space $I$. Since $J=I$, we can find $b \in I$, such that $a=\frac{1}{2}(bxe+b)$. Because $a$ is extremal, this forces $a=bxe=b$. Thus $axe=a$. Since this holds for every extremal point $a \in I$, and since $I$ is the closed convex hull of its extremal points (Krein-Millmann theorem), we conclude that $axe=a$ for all $a \in I$. In particular, we have $exe=e$ as we wanted.
\end{proof}

Recall that if $A$ is a unital $C^*$-algebra and $\Phi : A \rightarrow A$ is an idempotent UCP map, then $\Phi(A)$ equipped with the Choi-Effros product $(x,y) \mapsto \Phi(xy)$ is a $C^*$-algebra \cite[Theorem 3.1]{CE77}. By combining this fact with Theorem \ref{minimal idempotent}, one can associate to each compact convex semigroup of UCP maps on a von Neumann algebra a canonical ``boundary" $C^*$-algebra. See \cite{Si15} for examples.

\begin{cor}
Let $M$ be a von Neumann and $K \subset \mathrm{UCP}(M)$ a closed convex semigroup. Let $\Phi$ and $\Psi$ be two minimal idempotents. Then $\Psi|_{\Phi(M)} : \Phi(M) \rightarrow \Psi(M)$ is an isomorphism of $C^*$-algebras with inverse $\Phi|_{\Psi(M)}$.
\end{cor}
\begin{proof}
By Theorem \ref{minimal idempotent}, we have $\Psi \circ \Phi \circ \Psi=\Psi$ and $\Phi \circ \Psi \circ \Phi=\Phi$. Terefore $\Psi|_{\Phi(M)}$ and $\Phi|_{\Psi(M)}$ are UCP maps which are inverse to each others. Thus they must be isomorphisms of $C^*$-algebras.
\end{proof}

\begin{prop} \label{faithful idempotent}
Let $M$ be a von Neumann and $K \subset \mathrm{UCP}(M)$ a closed convex semigroup. Let $M^K=\{ x \in M \mid \forall \Psi \in K, \: \Psi(x)=x \}$. Let $\Phi \in K$ be a minimal idempotent. If $\Phi$ is faithful, then $M^K$ is a $C^*$-subalgebra of $M$ and $\Phi$ is a conditional expectation of $M$ onto $M^K$.
\end{prop}
\begin{proof}
For every $x \in \Phi(M)$, we have $\Phi(x^*x) \geq \Phi(x^*)\Phi(x)=x^*x$. But we have $\Phi(\Phi(x^*x) - x^*x) =0$. Thus $\Phi(x^*x)=x^*x$ because $\Phi$ is faithful. In particular, $x^*x \in \Phi(M)$ for all $x \in \Phi(M)$. By the polarization identity, this means that $\Phi(M)$ is a subalgebra of $M$. Then $\Phi$ is a faithful conditional expectation onto $\Phi(M)$. Now, take a unitary $u \in \Phi(M)$. For every $\Psi \in K$, we have $\Phi(\Psi(u))=u$ by Theorem \ref{minimal idempotent}. Since $\Phi$ is a faithful conditional expectation, this forces $\Psi(u)=u$. Thus $u \in M^K$. Since this holds for every unitary $u \in \Phi(M)$, we conclude that $\Phi(M)=M^K$.
\end{proof}

\section{The weak relative Dixmier property}

Let $N \subset M$ be an inclusion of von Neumann algebras. Let $\mathcal{D}(N,M) \subset \mathrm{UCP}(M)$ be the closed convex semigroup generated by $\Ad(u)$ for all $u \in \mathcal{U}(N)$. Note that $M^{\mathcal{D}(N,M)}=N' \cap M$. Observe that the inclusion $N \subset M$ has the weak relative Dixmier property if and only if $\mathcal{D}(N,M)$ contains a conditional expectation onto $N' \cap M$.

We need the following key lemma which relies on \cite[Theorem D]{Ma18}.
\begin{lem} \label{dixmier III1}
Let $N$ be a type $\III_1$ factor with separable predual. Let $\varphi$ be a faithful normal state on $N$. Then $\varphi \in \mathcal{D}(N,N)$.
\end{lem}
\begin{proof}
By \cite[Theorem D]{Ma18}, $N$ contains an irreducible subfactor with expectation $P \subset N$ such that $P$ is AFD of type $\III_1$. Let $\psi$ be a faithful normal state on $N$ such that $P$ is globally invariant by $\sigma^\psi$ and $P_\psi' \cap P=\C$. Since $P$ is AFD and irreducible in $N$, we know that $\mathcal{D}(P,N)$ contains a state $\phi$. Note that $\phi=\phi \circ E^\psi_P$. Thus, we have $\phi \circ E^\psi_{P_\psi' \cap N}=\psi$. Since $P_\psi$ centralizes $\psi$, we have $E^\psi_{P_\psi' \cap N} \in \mathcal{D}(P_\psi,N)$. We conclude that $\psi=\phi \circ E^\psi_{P_\psi' \cap N} \in \mathcal{D}(P,N) \subset \mathcal{D}(N,N)$. Finally, by Connes-St\o rmer transitivity theorem, we get $\varphi \in \mathcal{D}(N,N)$. 
\end{proof}

We are now ready to prove our main theorem.

\begin{proof}[Proof of the main theorem] First note that we can always write $N$ as an increasing union of a net of subalgebras $(N_i)_{i \in I}$ with separable predual which are globally invariant under $\sigma^\varphi$. Thus, we may assume that $N$ has a separable predual. Then, we may also assume that $M$ has a separable predual. Since $\mathcal{Z}(N)$ has the weak relative dixmier property in $M$, we may replace $M$ by $\mathcal{Z}(N)' \cap M$ and reduce to the case where $\mathcal{Z}(N) \subset \mathcal{Z}(M)$. Now by using the desintegration theory, we may assume that $N$ is a factor. If $N$ is semifinite, the result is already known. So we may assume that $N$ is a type $\III$ factor. If $N$ is not of type $\III_1$, then $N \cong N_0 \rtimes \Z$ for some semifinite von Neumann algebra with expectation $N_0 \subset N$. Since $N_0$ has the weak relative dixmier property in $M$ and $\Z$ is amenable, then $N$ also has the weak relative dixmier property in $M$. 

Finally, we deal with the case where $N$ is of type $\III_1$. By compactness of $\mathcal{D}(N,M)$, any element of $\mathcal{D}(N,N)$ can be extended to an element of $\mathcal{D}(N,M)$. In particular, thanks to Lemma \ref{dixmier III1}, we can find $\Psi \in \mathcal{D}(N,M)$ such that $\Psi |_N=\varphi$. Take $\Phi_0$ a minimal idempotent in $\mathcal{D}(N,M)$ and let $\Phi=\Phi_0 \circ \Psi$ which is again a minimal idempotent of $\mathcal{D}(N,M)$. Observe that $E$ commutes with every element of $\mathcal{D}(N,M)$. Thus, we have 
$$ E \circ \Phi = \Phi \circ E= \Phi_0 \circ \Psi \circ E=\Phi_0 \circ \varphi= \varphi.$$
In particular, $\Phi$ is faithful and by Proposition \ref{faithful idempotent}, we conclude that $\Phi$ is a conditional expectation on $N' \cap M$ (in fact $\Phi$ is the unique $\varphi$-preserving conditional expectation onto $N' \cap M$).
\end{proof}

\begin{rem}
Let $M \subset \B(H)$ be a von Neumann algebra. Let $\Phi \in \mathcal{D}(M,\B(H))$ be a minimal idempotent. Then $A_\Phi=\Phi(\B(H))$ (whose isomorphism class as a $C^*$-algebra does not depend on the choice of $\Phi$) contains $M'$. We have $A_\Phi=M'$ if and only if $M$ is injective. In general, $A_\Phi$ contains the injective enveloppe of $M'$ but it is not clear if they are equal or not. Notice that if $M \subset N \subset \B(H)$, then $M \subset N$ has the weak relative Dixmier property if and only if $A_\Phi \cap N=M' \cap N$.
\end{rem}

\bibliographystyle{plain}

\end{document}